\documentclass[11pt, reqno]{amsart}

\usepackage[english]{babel}
\usepackage[left=3.8cm, right=3.8cm, top=3.5cm, bottom=3.5cm]{geometry}
\usepackage{url, amsfonts, amsmath, amsthm, amssymb, mathtools, mathrsfs, enumerate}
\usepackage{color, xcolor, verbatim, tensor, tikz, tikz-cd}
\usetikzlibrary{angles,quotes}

\usepackage{etoolbox}

\usepackage{hyperref}
\hypersetup{colorlinks=true, linkcolor=blue, citecolor=cyan, urlcolor=magenta, linktoc=all}

\usetikzlibrary{matrix,calc}
\definecolor{wine-stain}{rgb}{0.5,0,0}

\newcommand{\red}[1]{\textcolor{red}{#1}}

\newtheorem{thm}{Theorem}[section]

\newtheorem{lem}[thm]{Lemma}
\newtheorem{cor}[thm]{Corollary}

\theoremstyle{definition}

\newtheorem{rem}[thm]{Remark}

\newtheorem*{ques*}{Question}
\newtheorem*{thm*}{Theorem}
\newtheorem*{rem*}{Remark}
\newtheorem*{rems*}{Remarks}
\newtheorem*{exs*}{Examples}
\newtheorem*{mthm*}{Main Theorem}

\numberwithin{equation}{section}

\newcommand{\la}{\lambda}

\makeatletter
\renewcommand*{\eqref}[1]{%
	\hyperref[{#1}]{\textup{\tagform@{\ref*{#1}}}}%
}
\makeatother

\title[Positivity of $Q$-curvature via the continuity method]{A note on the positivity of $Q$-curvature via the continuity method}
\author[Ramesh Mete]{Ramesh Mete}
\address{Department of Mathematics, Indian Institute of Technology Bombay, Powai, Mumbai - 400076, India.}
\email{ramesh2025m@gmail.com, rameshm@math.iitb.ac.in}
\date{\today}

\begin{document}
	
\maketitle

\begin{abstract}
Suppose $(M, g)$ is a smooth, closed, $5$-dimensional Riemannian manifold with positive Yamabe invariant $Y(M, [g]) > 0$ and positive Yamabe-type $Q$-curvature invariant $Y_{4}^{\ast}(M, [g]) > 0$. Using the continuity method, we prove the existence of a metric in the conformal class $[g]$ with positive scalar curvature and positive $Q$-curvature, assuming an additional condition on an ``initial metric" in the class $[g]$.
\end{abstract}

\vspace*{2mm}
\section{Introduction}

The $Q$-curvature and its associated fourth-order Paneitz operator have become one of the central objects of study in conformal geometry. Compared to classical curvature notions, Branson's $Q$-curvature and its higher-order analogues are more recent development. In this paper, we use the method of continuity to establish the positivity of $Q$-curvature and scalar curvature on five-dimensional closed Riemannian manifolds under certain conditions.

\subsection{Fourth-order $Q$-curvature and the corresponding Paneitz operator}

Let $(M, g)$ be a closed (i.e. compact and without boundary) Riemannian manifold of dimension $n\geq 3$. If $\mathrm{Ric}_g$ and $R_g$ denote the Ricci tensor and scalar curvature of the Riemannian metric $g$ respectively, then the Schouten tensor is defined by
\begin{align*}
A_g := \frac{1}{n-2}\left(\mathrm{Ric}_g - \frac{R_g}{2(n-1)} g\right).
\end{align*}
Throughout the paper, we adopt the notation
$$J_g:= \frac{R_g}{2(n-1)}.$$
Then the $Q$-curvature of Branson \cite{Branson1985} for the metric $g$ is defined as
\begin{equation}\label{eq:Q-curv-defn-involving-modulus-of-Schouten}
Q_g := - \Delta_{g} J_g - 2|A_g|^2_{g} + \frac{n}{2} J^2_{g}.
\end{equation}
Here $\Delta_{g} := \mathrm{div}_{g}(\nabla_g)$ is the Laplace-Beltrami operator.

\vspace*{1.5mm}
Viaclovsky \cite{Viaclovsky2000} first introduced the $\sigma_k$-scalar curvature $\sigma_{k}(A_g)$ and studied the associated $\sigma_k$-Yamabe problem. Suppose $\sigma_k : \mathbb{R}^n \longrightarrow\mathbb{R}$ is the $k$-th elementary symmetric polynomial for any positive integer $k$. Then $\sigma_k(A_g)$ denotes the function $\sigma_k$ applied to the eigenvalues of $A_g$ with respect to the metric $g$. It is easy to see that $\sigma_1(A_g) = J_g$. Since we also know that
\begin{align*}
\sigma_2(A_g) = \frac{1}{2}(J^2_{g} - |A_g|^2_{g}),
\end{align*}
the $Q$-curvature defined above can be re-written as follows:
\begin{equation}\label{eq:Q-curv-defn-involving-sigma_2}
Q_g = - \Delta_{g} J_{g} + \frac{n-4}{2}J^2_{g} + 4\sigma_2(A_g).
\end{equation}
More explicitly, in terms of the scalar curvature and the Ricci tensor, we derive the following expression (e.g., see \cite{HangYang2016-lec-notes})
\begin{align*}
Q_g = -\frac{1}{2(n-1)}\Delta_{g}R_g - \frac{2}{(n-2)^2}|\mathrm{Ric}_g|^2_{g} + \frac{n^2(n-4) + 16(n-1)}{8(n-1)^2 (n-2)^2} R^2_{g}.
\end{align*}

\vspace*{1.5mm}
We now recall the definition of the Paneitz operator \cite{Paneitz1983} associated to Branson's $Q$-curvature. It is an elliptic self-adjoint differential operator with the bi-Laplacian $\Delta^2_{g}$ as the leading term. More precisely, the Paneitz operator, denoted by $P_g$, is defined as
\begin{align*}
P_g \varphi := \Delta^2_{g}\varphi + \mathrm{div}\left(4A_{g}(\nabla\varphi, e_i)e_i - (n-2)J_g \nabla\varphi\right) + \frac{n-4}{2} Q_g \varphi,
\end{align*}
where $\{e_1, \cdots, e_n\}$ is a local orthonormal frame with respect to $g$. For dimension $n \neq 4$, the Paneitz operator satisfies the following conformal transformation law:
\begin{equation}\label{eq:conf-transform-of-Paneitz-dim-at-least-five}
P_{\rho^{\frac{4}{n-4}}g}(\varphi) = \rho^{- \frac{n+4}{n-4}} P_g(\rho\varphi),
\end{equation}
where $\rho\in C^\infty(M)$ with $\rho > 0$. In particular, for $n \neq 4$ we observe that
\begin{equation*}
Q_{\rho^{\frac{4}{n-4}}g} = \frac{2}{n-4}\rho^{- \frac{n+4}{n-4}} P_g(\rho).
\end{equation*}
On the other hand, in dimension $n=4$ we have 
\begin{align*}
P_{e^{2w} g}(\varphi) = e^{-4 w} P_g(\varphi), ~~~ Q_{e^{2w} g} = e^{-4 w}(P_g w + Q_g)
\end{align*}
for any $w \in C^\infty(M)$.

\vspace*{1.5mm}
\textbf{Notation:} For brevity we will denote the scalar curvature, Ricci tensor, Schouten tensor, Laplace-Beltrami operator, $Q$-curvature, and Paneitz operator for the Riemannian metric $g$ by $R$, $\mathrm{Ric}$, $A$, $\Delta$, $Q$ and $P$ respectively. The corresponding quantities for any conformal metric $\tilde g := u^{\frac{4}{n-4}} g$ will be denoted by $\tilde R$, $\widetilde{\mathrm{Ric}}$, $\tilde A$, $\tilde\Delta$, $\tilde Q$ and $\tilde P$ respectively. Moreover, we simply write $J$ for the quantity $J_g$ and we write $\tilde J$ for $J_{\tilde g}$.

\subsection{Some conformally invariant quantities and known results}

Suppose $(M, g)$ is a smooth Riemannian manifold of dimension $n \geq 3$. We denote by $\mathrm{Vol}_{g}(M)$ the volume of $(M, g)$. The Yamabe invariant is defined by (cf. \cite{LeePar1987})
\begin{equation}\label{eq:Yam-inv}
Y(M, [g]) := \underset{\tilde g \in [g]}{\inf} \frac{\int_M \tilde R \ d\mu_{\tilde g}}{\mathrm{Vol}_{\tilde g}(M)^{\frac{n-2}{n}}} = \inf_{u\in C^\infty(M),~ u>0} \frac{\int_M L(u)u \ d\mu_g}{\Vert u \Vert_{L^{\frac{2n}{n-2}}}^{2}},
\end{equation}
where the conformal Laplacian $L$ associated to the metric $g$ is
\begin{equation}\label{eq:conformal-Laplacian-defn}
L := - \frac{4(n-1)}{n-2}\Delta + R.
\end{equation}

\vspace*{1.5mm}
Next, the Yamabe-type invariant corresponding to the fourth order $Q$-curvature is given by (cf. \cite{HangYang2016-CPAM})
\begin{align*}
Y_{4}^{+}(M, [g]) := \frac{n-4}{2} \underset{\tilde g \in [g]}{\inf} \frac{\int_M \tilde Q \ d\mu_{\tilde g}}{\mathrm{Vol}_{\tilde g}(M)^{\frac{n-4}{n}}} = \inf_{u\in C^\infty(M),~ u>0} \frac{\int_M P(u)u \ d\mu_g}{\Vert u \Vert_{L^{\frac{2n}{n-4}}}^{2}}.
\end{align*}
Assuming the Yamabe invariant $Y(M, [g]) > 0$, another closely related conformally invariant quantity $Y^{\ast}_{4}(M, [g])$ was introduced in \cite{GurHangLin2016} and it is defined as
\begin{align*}
Y_{4}^{\ast}(M, [g]) := \frac{n-4}{2} \inf_{\underset{\tilde{R} > 0}{\tilde g \in [g]}} \frac{\int_M \tilde Q \ d\mu_{\tilde g}}{\mathrm{Vol}_{\tilde g}(M)^{\frac{n-4}{n}}}.
\end{align*}
From the above definitions, we easily see that $Y_{4}^{\ast}(M, [g]) \geq Y_{4}^{+}(M, [g])$.

\vspace*{1.5mm}
We consider another conformally invariant quantity related to the $\sigma_{2}$-curvature, denoted by $Y_{\sigma_2}(M, [g])$ (also written as $Y_{2}(M, [g])$ following \cite{GeWangWei2026arXiv}). Using our notation, it is defined as follows
\begin{align*}
Y_{\sigma_2}(M, [g]) := \inf_{\underset{\tilde{R} > 0}{\tilde g \in [g]}} \frac{\int_M \sigma_{2}(\tilde A) d\mu_{\tilde g}}{\mathrm{Vol}_{\tilde g}(M)^{\frac{n-4}{n}}}.
\end{align*}

\vspace*{1.5mm}
The total $Q$-curvature plays an important role on four-dimensional Riemannian manifolds. If $(M^4, g)$ is a closed Riemannian manifold of dimension $4$, integrating \eqref{eq:Q-curv-defn-involving-sigma_2} we have the following relation
$$\int_{M^4} Q\ d\mu_g = 4 \int_{M^4} \sigma_{2}(A)\ d\mu_g.$$
But recall that the Chern-Gauss-Bonnet formula on $M^4$ is the following: (cf. \cite{Gur1998,ChaGurYang2002a})
$$ 4 \int_{M^4} \sigma_{2}(A)\ d\mu_g + \frac{1}{4}\int_{M^4} |W|^2_{g} \ d\mu_g = 8\pi^2 \chi(M^4), $$
where $W$ denotes the Weyl curvature tensor of $g$ and $\chi(M^4)$ denotes the Euler characteristic of the four-manifold $M^4$. Therefore, 
the following identity holds:
\begin{align*}
\int_{M^4} Q\ d\mu_g + \frac{1}{4}\int_{M^4} |W|^2_{g} \ d\mu_g = 8\pi^2 \chi(M^4),
\end{align*}
and hence the total $Q$-curvature $\mathcal{Q}(M^4, [g]):= \int_{M^4} Q_g \ d\mu_g$ is a conformally invariant quantity on any four-manifold ($M^4, g)$. The positivity of $\mathcal{Q}(M^4, [g])$ yields pointwise positivity of other curvature notions, provided the Yamabe invariant $Y(M^4, [g])$ (defined in \eqref{eq:Yam-inv}) is also positive. Chang, Gursky, and Yang \cite{ChaGurYang2002a} proved that if $\mathcal{Q}(M^4, [g])>0$ and $Y(M^4, [g]) >0$ on a closed $4$-manifold $(M^4, g)$, then there exists a conformal metric $\tilde g \in [g]$ such that $\sigma_{2}(\tilde A) > 0$ and hence $\widetilde{\mathrm{Ric}}> 0$. As a consequence, the first de Rham cohomology group $H^1_{dR}(M)=0$ (see also \cite{Gur1998}). Recently, Lee \cite{Lee2025} showed that the same positivity conditions $\mathcal{Q}(M^4, [g])>0$ and $Y(M^4, [g]) >0$ together imply the existence of a metric in the conformal class $[g]$ with constant positive $Q$-curvature and positive scalar curvature. The problem of finding conformal metrics with constant $Q$-curvature in four dimensions has been resolved. The reader is referred to \cite{ChaYang1995, DjadliMal2008, LiLiLiu2012, Lee2025} and the references therein for further details.

\vspace*{1.5mm}
We now review relevant results concerning the $Q$-curvature on Riemannian manifolds of dimension five and higher. Assuming pointwise positivity of both the scalar curvature and the $Q$-curvature for $n \geq 5$, Gursky and Malchiodi \cite{GurMal2015} established the following result on the existence of a conformal metric with constant positive $Q$-curvature.

\begin{thm}[{\cite[Theorem D]{GurMal2015}}]
Let $(M^n, g)$ be a closed Riemannian manifold of dimension $n \geq 5$ satisfying 
\begin{equation}\label{eq:GurMal-semipositive-Q-curv-condition}
Q_{g} \geq 0, ~~ Q_g \not\equiv 0, ~~ R_g \geq 0.
\end{equation}
Then there is a conformal metric $\tilde{g}:= u^{4/(n-4)}g$ with positive scalar curvature and constant positive $Q$-curvature.
\end{thm}

Under assumption \eqref{eq:GurMal-semipositive-Q-curv-condition}, Gursky and Malchiodi also proved that the Paneitz operator $P_g$ is positive and satisfies the strong maximum principle, and its Green's function $G_{P}$ is positive. Note that condition \eqref{eq:GurMal-semipositive-Q-curv-condition} guarantees the strict positivity of the scalar curvature (cf. \cite[Lemma 2.1]{GurMal2015}). Motivated by these results and their work in \cite{HangYang2015-sign-of-Green}, Hang and Yang \cite[Theorem 1.1]{HangYang2016-CPAM} considered the following condition
\begin{equation}\label{eq:HangYang-semipositive-Q-curv-and-positive-Yam-condition}
Q_{g} \geq 0, ~~ Q_g \not\equiv 0, ~~ Y(M, [g]) > 0
\end{equation}
for dimension $n \geq 5$, and they proved that the Green's function of $P_g$ is positive and that there exists a conformal metric $\tilde g$ with $\tilde Q = 1$. Later, Gursky, Hang, and Lin \cite{GurHangLin2016} proved the following result in dimensions $n \geq 6$, assuming the strict positivity of the associated conformally invariant quantities.

\begin{thm}[{\cite[Theorem 1.1]{GurHangLin2016}}]
\label{thm:GurHangLin-dim-atleast-6-cont-method}
Let $(M, g)$ be a smooth compact Riemannian manifold with dimension $n\geq 6$. If $Y(M, [g]) > 0$ and $Y_{4}^{\ast}(M, [g]) > 0$, then there exists a metric $\tilde{g}\in[g]$ satisfying $\tilde{R} > 0$ and $\tilde{Q} > 0$.
\end{thm}

Gursky, Hang, and Lin \cite{GurHangLin2016} proved their main result (Theorem \ref{thm:GurHangLin-dim-atleast-6-cont-method}) using the continuity method; specifically, they set up a continuity path \eqref{eq:GurHangLin-continuity-path-in-terms-of-t} for conformal metrics in $[g]$ and proved certain \emph{a priori} estimates. They conjectured that the same result holds for dimension $n=5$. Recently, Li \cite[Theorem 1.1]{Li2026arXiv-GHL-conj-dim-5} affirmatively resolved this conjecture without using the continuity method, building on earlier work by Hang and Yang \cite{HangYang2015-sign-of-Green, HangYang2016-lec-notes}.

\subsection{Main result}

The main goal of this paper is to establish the Gursky-Hang-Lin conjecture via the continuity method under an additional criterion. Specifically, we assume the existence of an initial conformal metric satisfying a certain positivity condition. Our main result is stated below.

\begin{thm}\label{thm:GurHangLin-conj-dim-5-cont-method-under-an-extra-cond}
Let $(M, g)$ be a smooth closed Riemannian manifold of dimension $n=5$. Suppose $Y(M, [g]) > 0$ and $Y_{4}^{\ast}(M, [g]) > 0$. Assume further that there exists a conformal metric $g_0 \in [g]$ and a constant $t_0 \in (1, \frac{10}{3}]$ such that
\begin{equation}\label{eq:extra-cond-on-initial-metric-for-nonempty-dim-5}
t_0 \left(- \Delta_{0} J_0 + \frac{n-4}{2} J_{0}^2\right) + 4\sigma_2(A_0) > 0 \quad \text{and} \quad J_0 > 0.
\end{equation}
Here $\Delta_{0} := \Delta_{g_0}$, $J_0 := J_{g_0}$ and $A_0 := A_{g_0}$. Then there exists a conformal metric $\tilde{g}\in[g]$ with positive scalar curvature $\tilde{R} > 0$ and positive $Q$-curvature $\tilde{Q} > 0$.
\end{thm}

We have the following corollary.

\begin{cor}\label{cor:positivity-of-Q-curv-assump-sigma-2-Yam-inv-positive}
Let $(M, g)$ be a smooth closed $5$-dimensional Riemannian manifold. Suppose $Y(M, [g]) > 0$ and $Y_{\sigma_{2}}(M, [g]) > 0$. Assume further that there exists a conformal metric $g_0 \in [g]$ and a constant $t_0 \in (1, \frac{10}{3}]$ such that \eqref{eq:extra-cond-on-initial-metric-for-nonempty-dim-5} holds. Then there exists a conformal metric $\tilde{g}\in[g]$ satisfying $\tilde{Q} > 0$ and $\tilde{R} > 0$.
\end{cor}

We make a few remarks regarding the additional pointwise positivity condition \eqref{eq:extra-cond-on-initial-metric-for-nonempty-dim-5}. First, by using the formula \eqref{eq:Q-curv-defn-involving-sigma_2}, this criterion is equivalent to
\begin{align*}
Q_0 - \lambda_{0} \sigma_{2}(A_0) > 0 \quad \text{and} \quad J_0 > 0,
\end{align*}
where the positive constant $\lambda_{0}$ is defined as $$ \lambda_0 := \frac{4(t_0 - 1)}{t_0}. $$
Since $1 < t_0 \leq \frac{10}{3}$ by assumption \eqref{eq:extra-cond-on-initial-metric-for-nonempty-dim-5}, it follows that $$0 < \lambda_{0} \leq  \frac{14}{5}.$$
Note that if $\sigma_2(A_0) > 0$ and \eqref{eq:extra-cond-on-initial-metric-for-nonempty-dim-5} holds, the assertion follows immediately because $g_0$ itself is a conformal metric with positive $Q$-curvature and positive scalar curvature. Otherwise, Theorem \ref{thm:GurHangLin-conj-dim-5-cont-method-under-an-extra-cond} yields a conformal metric $\tilde g \in [g]$ with $\tilde Q > 0$ and $\tilde R >0$, provided the $Q$-curvature of the initial metric $g_0$ satisfies the lower bound $Q_0  > \lambda_{0} \sigma_{2}(A_0)$. Next, by \cite[Proposition 2.1]{GurHangLin2016}, if $Y(M, [g])>0$ for a smooth closed Riemannian manifold $(M, g)$ of dimension $n \geq 5$, then there exists a conformal metric $g_0 \in [g]$ satisfying
\begin{equation*}
- \Delta_{0} J_0 + \frac{n-4}{2} J_{0}^2 > 0 \quad \text{and} \quad J_0 > 0.
\end{equation*}
In particular, one can choose $t_0 \gg 1$ sufficiently large such that
\begin{align*}
t_0\left(- \Delta_{0} J_0 + \frac{n-4}{2} J_{0}^2\right) + 4\sigma_2(A_0) > 0.
\end{align*}
Lastly, we expect the extra condition \eqref{eq:extra-cond-on-initial-metric-for-nonempty-dim-5} to be removed even via the continuity method approach for the Gursky-Hang-Lin conjecture. This expectation is motivated by recent work of Li \cite{Li2026arXiv-GHL-conj-dim-5}, who established this conjecture for $n=5$ via an alternative approach based on \cite{HangYang2015-sign-of-Green, HangYang2016-lec-notes} rather than the continuity method. One might need to set up a new continuity path, different from \eqref{eq:GurHangLin-continuity-path-in-terms-of-t}, to remove this extra condition \eqref{eq:extra-cond-on-initial-metric-for-nonempty-dim-5}. This would provide an alternative proof of the Gursky-Hang-Lin conjecture via the continuity method.

\vspace*{1.5mm}
We finish the introduction with one final remark. On a closed Riemannian manifold $(M, g)$ of dimension $n\geq 5$, Ge, Wang, and Wei \cite{GeWangWei2026arXiv} recently proved that if $Y(M, [g]) > 0$ and $Y_{\sigma_{2}}(M, [g]) > 0$ (as in Corollary \ref{cor:positivity-of-Q-curv-assump-sigma-2-Yam-inv-positive}), then $Y_{\sigma_{2}}(M, [g])$ is achieved by a metric $\hat g \in [g]$ satisfying $\hat R > 0$ and $\sigma_{2}(\hat A) > 0$. Here, we write $\hat g = \hat{u}^{\frac{4}{n-2}} g$ for a smooth function $\hat{u} > 0$ on $M$. Furthermore, the $Q$-curvature $\hat Q$ of $\hat g$ is positive if $\hat u$ solves the sub-critical equation
\begin{equation*}
L \phi = \phi^{p},
\end{equation*}
where $L$ is the conformal Laplacian with respect to $g$ (defined in \eqref{eq:conformal-Laplacian-defn}) and $$ \max\left(1, \frac{6}{n-2}\right) < p < \frac{n+2}{n-2}.$$

\vspace*{2mm}
\section{Setup of the continuity method}

Let $(M, g)$ be a smooth, closed Riemannian manifold of dimension $n \geq 5$. For any $t\in [1, t_0]$ with $t_0 \leq  \frac{10}{3}$, we consider the following continuity path given by \cite[Equation (1.22)]{GurHangLin2016}:
\begin{equation}\label{eq:GurHangLin-continuity-path-in-terms-of-t}
t\left(- \tilde\Delta \tilde{J} + \frac{n-4}{2}\tilde{J}^2\right) + 4\sigma_{2}(\tilde A) = f u^{-\frac{n+4}{n-4}}, ~~~~ \tilde{g} := u^{\frac{4}{n-4}}g.
\end{equation}
Here the smooth real-valued function $f$, which is independent of $t$, is defined as
\begin{align*}
f := u_{0}^{\frac{n+4}{n-4}}\left(t_{0}\left(-\Delta_{0}J_{0} + \frac{n-4}{2} J_{0}^2\right) + 4\sigma_2(A_0)\right),
\end{align*}
where $\Delta_{0}$, $J_0$, and $A_0$ are with respect to $g_0 := u_{0}^{\frac{4}{n-4}}g$. Note that $f$ is positive by the assumption \eqref{eq:extra-cond-on-initial-metric-for-nonempty-dim-5}. Then we consider the following set
\begin{align*}
\mathcal{S} := \Big\{ t\in [1, t_0]:~ \exists ~ u\in C^{\infty}(M),~ u > 0, ~ \text{solving}~ \eqref{eq:GurHangLin-continuity-path-in-terms-of-t}~\text{with}~ \tilde{R} > 0\Big\}.
\end{align*}
Our goal is to show that $\mathcal{S} = [1, t_0]$. Since $t_0 \in \mathcal{S}$ by the definition of $f$ and assumption \eqref{eq:extra-cond-on-initial-metric-for-nonempty-dim-5}, it suffices to show that $\mathcal{S}$ is both open and closed.

\vspace*{1.5mm}
Using the formula \eqref{eq:Q-curv-defn-involving-sigma_2}, the continuity path \eqref{eq:GurHangLin-continuity-path-in-terms-of-t} is equivalent to
\begin{equation}\label{eq:GurHangLin-continuity-path-in-terms-of-lambda}
\tilde{Q} - \lambda \sigma_2(\tilde A) = \chi u^{- \frac{n+4}{n-4}}, ~~~~ \tilde{g} := u^{\frac{4}{n-4}}g,
\end{equation}
where the constant $\lambda$ and the positive smooth function $\chi$ are given by respectively
\begin{align*}
\lambda := \frac{4(t - 1)}{t} \quad \text{and} \quad  \chi := \frac{f}{t}.
\end{align*}
Since $t\in [0, t_0]$, we have $0\leq \lambda \leq \lambda_{0} \leq \frac{14}{5}$. Now, we define the following set
\begin{align}\label{eq:defn-of-the-set-Sigma}
\Sigma := \Big\{\lambda\in[0, \lambda_{0}]:~ \exists~ u \in C^{\infty}(M), ~ u>0, ~ \text{solving} ~ \eqref{eq:GurHangLin-continuity-path-in-terms-of-lambda}~ \text{with}~ \tilde{J}> 0 \Big\}.
\end{align}
Because $\lambda \in \Sigma$ if and only if $t\in\mathcal{S}$, the fact that $t_0 \in \mathcal{S}$ implies that $\Sigma$ is non-empty.

\vspace*{2mm}
\section{Proof of the main result}

In this section, we show that the set $\Sigma$ (defined in \eqref{eq:defn-of-the-set-Sigma}) is both open and closed under the assumptions of Theorem \ref{thm:GurHangLin-conj-dim-5-cont-method-under-an-extra-cond}.

\subsection{Openness of the set $\Sigma$}

Note that $\tilde{g} = u^{\frac{4}{n-4}}g$ solves \eqref{eq:GurHangLin-continuity-path-in-terms-of-lambda} if and only if $\mathcal{N}[u]=0$, where
\begin{align*}
\mathcal{N}[u]:= \tilde{Q} - \lambda \sigma_2(\tilde A) - \chi u^{- \frac{n+4}{n-4}}.
\end{align*}
The linearization of the map $\mathcal{N}$ at the point $u \in C^{\infty}(M)$, $u> 0$, is given by
\begin{align*}
\tilde{S}\varphi := \frac{d}{dt}\Big|_{t=0}\mathcal{N}[u+t\varphi] = \frac{2}{n-4} \tilde{H}(u^{-1}\varphi),
\end{align*}
where the operator $\tilde H$ with respect to the metric $\tilde g = u^{\frac{4}{n-4}}g$ is defined as
\begin{align*}
\tilde{H}(\psi) := \tilde{P}\psi - \frac{n+4}{2}\tilde{Q}\psi + \lambda \left(\tilde{J}\tilde\Delta\psi - \tilde{g}(\tilde A, \tilde\nabla^{2}\psi) + 4 \sigma_{2}(\tilde A)\psi\right) + \frac{n+4}{2}\chi u^{-\frac{n+4}{n-4}}\psi.
\end{align*}
We have the following result about the positivity of this operator.

\begin{lem}[{\cite[Lemma 2.1]{GurHangLin2016}}]
\label{lem:positivity-of-operator-H-for-dim-at-least-6}
Let $(M, g)$ be a smooth compact Riemannian manifold with dimension $n \geq 6$. If $0\leq \lambda \leq 4$, and
\begin{equation}\label{eq:positivity-cond-Q-minus-la-times-sigma2}
Q - \lambda \sigma_{2}(A) > 0, ~~~ J>0,
\end{equation}
then the operator
\begin{align}\label{eq:defn-of-operator-H}
H(\varphi) := P\varphi - \frac{n+4}{2}Q\varphi + \lambda\left(J\Delta\varphi - g(A, \nabla^{2}\varphi) + 4\sigma_{2}(A)\varphi\right) + \frac{n+4}{2}(Q - \lambda\sigma_{2}(A))\varphi
\end{align}
is positive definite.
\end{lem}

As mentioned in \cite{GurHangLin2016}, Lemma \ref{lem:positivity-of-operator-H-for-dim-at-least-6} along with the implicit function theorem and standard elliptic theory implies that the set $\Sigma$ is open when dimension $n \geq 6$.

\vspace*{1.5mm}
For dimension $n=5$ we have the following modified result about the positivity of the operator $H$.

\begin{lem}\label{lem:positivity-of-operator-H-for-dim-5}
Let $(M, g)$ be a smooth closed Riemannian manifold of dimension $n = 5$. If the positivity condition \eqref{eq:positivity-cond-Q-minus-la-times-sigma2} holds for $0\leq \lambda \leq \lambda_{0} \leq \frac{14}{5}$, then the operator $H$, defined in \eqref{eq:defn-of-operator-H}, is positive definite.
\end{lem}

\begin{proof}
For any smooth function $\varphi$ we have (see equation (2.25) in \cite{GurHangLin2016} with $n=5$)
\begin{equation}\label{eq:calculation-about-operator-H-in-dim-5}
\begin{split}
\int_M (H\varphi)\varphi\ d\mu &= \int_M (\Delta\varphi)^2\ d\mu + (3 - \lambda) \int_M J|\nabla\varphi|^{2}\ d\mu \\ &\hspace*{1cm} - (4 - \lambda)\int_M A(\nabla\varphi, \nabla\varphi)\ d\mu + \frac{1}{2}\int_M (Q - \lambda \sigma_{2}(A))\varphi^2\ d\mu.
\end{split}
\end{equation}
Since for $n=5$ we have
\begin{align*}
Q - \lambda \sigma_{2}(A) = - \Delta J - 2|A|^2 + \frac{5}{2}J^2 -\frac{\lambda}{2}(J^2 - |A|^2),
\end{align*}
the first positivity condition in \eqref{eq:positivity-cond-Q-minus-la-times-sigma2} implies that
\begin{equation}\label{equ-1-dim-5-openness-lem}
\frac{4-\lambda}{2} |A|^2 < - \Delta J + \frac{5-\lambda}{2} J^2.
\end{equation}
Now because $J > 0$, as in \cite[page 2145]{GurMal2015}, we have 
\begin{align*}
2 A(\nabla\varphi, \nabla\varphi) \leq \frac{|A|^2}{J}|\nabla\varphi|^2 + J |\nabla\varphi|^2.
\end{align*}
Using \eqref{equ-1-dim-5-openness-lem} the above inequality implies
\begin{equation*}
(4 - \lambda) A(\nabla\varphi, \nabla\varphi) < -\frac{\Delta J}{J}|\nabla\varphi|^2 + \frac{9 - 2\lambda}{2} J |\nabla\varphi|^2.
\end{equation*}
So, integrating over the manifold yields
\begin{equation}\label{equ-2-dim-5-openness-lem}
(4 - \lambda) \int_M A(\nabla\varphi, \nabla\varphi)\ d\mu < -\int_M \frac{\Delta J}{J}|\nabla\varphi|^2\ d\mu + \frac{9 - 2\lambda}{2} \int_M J |\nabla\varphi|^2\ d\mu.
\end{equation}
Next, as in \cite[page 2146]{GurMal2015}, for the first term on the right-hand side  we have
\begin{equation}\label{equ-3-dim-5-openness-lem}
\begin{split}
-\int_M \frac{\Delta J}{J}|\nabla\varphi|^2\ d\mu &\leq \int_M |\nabla^2\varphi|^2\ d\mu \\ &= \int_M \left((\Delta\varphi)^2 - 3A(\nabla\varphi, \nabla\varphi) - J |\nabla\varphi|^2\right)\ d\mu.
\end{split}
\end{equation}
Here the equality is basically the (integral) Bochner formula for dimension $n=5$. Combining \eqref{equ-2-dim-5-openness-lem} and \eqref{equ-3-dim-5-openness-lem} we obtain
\begin{equation*}
(7 - \lambda) \int_M A(\nabla\varphi, \nabla\varphi)\ d\mu < \int_M (\Delta\varphi)^2\ d\mu + \frac{7 - 2\lambda}{2} \int_M J |\nabla\varphi|^2\ d\mu.
\end{equation*}
Therefore,
\begin{equation}\label{equ-4-dim-5-openness-lem}
(4 - \lambda) \int_M A(\nabla\varphi, \nabla\varphi)\ d\mu < \frac{4-\lambda}{7-\lambda} \int_M (\Delta\varphi)^2\ d\mu + \frac{(4-\lambda)(7 - 2\lambda)}{2(7-\lambda)} \int_M J |\nabla\varphi|^2\ d\mu.
\end{equation}
Now using \eqref{equ-4-dim-5-openness-lem} in \eqref{eq:calculation-about-operator-H-in-dim-5} we see that
\begin{align*}
\int_M (H\varphi)\varphi\ d\mu &> \left(1 - \frac{4-\lambda}{7-\lambda}\right)\int_M (\Delta\varphi)^2\ d\mu + \frac{1}{2}\int_M (Q - \lambda \sigma_{2}(A))\varphi^2\ d\mu \\ &\hspace*{2cm} + \left((3-\lambda) - \frac{(4-\lambda)(7 - 2\lambda)}{2(7-\lambda)}\right) \int_M J |\nabla\varphi|^2\ d\mu \\
&= \frac{3}{7-\lambda} \int_M (\Delta\varphi)^2\ d\mu + \frac{1}{2}\int_M (Q - \lambda \sigma_{2}(A))\varphi^2\ d\mu
\\ &\hspace*{2cm} + \frac{14 - 5\lambda}{2(7-\lambda)} \int_M J |\nabla\varphi|^2\ d\mu.
\end{align*}
Since $\la\in[0, \frac{14}{5}]$, the right-hand side is non-negative. Hence, the operator $H$ is positive definite.
\end{proof}

As an implication of Lemma \ref{lem:positivity-of-operator-H-for-dim-5}, the openness of the set $\Sigma$ (see \eqref{eq:defn-of-the-set-Sigma}) also holds in dimension $n=5$.

\subsection{Apriori estimates and closedness of the set $\Sigma$}

It only remains to show that the set $\Sigma$ (defined in \eqref{eq:defn-of-the-set-Sigma}) is also closed in dimension $n=5$ for $\la_0 \leq \frac{14}{5}$.

\vspace*{1.5mm}
We first recall a result due to Gursky, Hang, and Lin \cite{GurHangLin2016} regarding a uniform positive lower bound and an \emph{a priori} estimate for smooth positive solutions to the equation \eqref{eq:GurHangLin-continuity-path-in-terms-of-lambda}.

\begin{lem}[{\cite[Lemma 3.1]{GurHangLin2016}}]
\label{lem:unif-positive-lower-bdd-GurHangLin}
Assume that $(M, g)$ is a smooth compact Riemannian manifold with dimension $n \geq 5$. If $Y(M, [g]) > 0$, $Y^{\ast}_{4}(M, [g]) > 0$, and $0\leq \lambda \leq \lambda_{0} < 4$, then any smooth positive solution $u$ to \eqref{eq:GurHangLin-continuity-path-in-terms-of-lambda} with $\tilde{J} > 0$ satisfies
\begin{equation}\label{eq:uniform-positive-lower-bdd-and-some-Lp-estm}
\Vert u \Vert_{L^{\frac{2n}{n-4}}} \leq c, ~~\text{and}~~ u \geq c > 0.
\end{equation}
Here the constant $c$ is independent of $u$ and $\lambda$.
\end{lem}

\begin{rem}
Gursky, Hang and Lin \cite{GurHangLin2016} stated the above lemma for $n \geq 6$, but their argument also holds for $n=5$ (in particular, for $\la_0 \leq  \frac{14}{5}$).
\end{rem}

To obtain higher-order estimates, Gursky, Hang, and Lin \cite{GurHangLin2016} introduced the following auxiliary function
\begin{equation}\label{eq:aux-func-v-GurHangLin}
v := u^{- q - 1}\left(-\Delta u + \frac{n-4}{2} J u\right),
\end{equation}
where $q\geq 0$ to be chosen later. In other words, it satisfies the following equation
\begin{equation}\label{eq:2nd-order-equ-satisfied-by-aux-func-v}
\Delta u = \frac{n-4}{2} J u - u^{q+1}v.
\end{equation}
But note that the positivity condition of scalar curvature $\tilde{J} > 0$ is equivalent to the following inequality (see either (2.9) in \cite{GurHangLin2016} or (2.7) in \cite{GurMal2015})
\begin{align*}
\Delta u < - \frac{2}{n-4} u^{-1}|\nabla u|^2 + \frac{n-4}{2} J u.
\end{align*}
Hence the function $v$ satisfies 
\begin{align*}
v > \frac{2}{n-4} u^{-q-2}|\nabla u|^2.
\end{align*}

\vspace*{1.5mm}
The main idea is to get the $L^\infty$ bound of the auxiliary function $v$ for certain choice of the constant $q\geq 0$. Using the function $v$ one can obtain the following integral inequality (see \cite[Equation (3.36)]{GurHangLin2016})
\begin{equation}\label{eq:key-integral-ineq-from-Sec3-in-GurHangLin}
\begin{split}
&\frac{\alpha}{2} \int_M v^{\alpha - 1} |\nabla v|^2\ d\mu + \min(A_{\alpha, q, n}, B_{\alpha, q, n}) \int_M u^q v^{\alpha + 2}\ d\mu \\ &\leq c \int_M v^{\alpha + 1}\ d\mu + c\alpha\int_M v^{\alpha}\ d\mu
\end{split}
\end{equation}
for $\alpha \geq 1$ large enough, where $c>0$ is the uniform constant given in \eqref{eq:uniform-positive-lower-bdd-and-some-Lp-estm}, and the constants $A_{\alpha, q, n}$ and $B_{\alpha, q, n}$ are given by respectively
\begin{align*}
A_{\alpha, q, n} &:= \frac{(\alpha - 1)(q+1)}{\alpha + 1} - \frac{(n-1)\lambda}{n(n-4)}; \\
B_{\alpha, q, n} &:= (q + 1) \left(-\frac{n-4}{2}q + \frac{\alpha - n + 3}{\alpha + 1}\right).
\end{align*}
After getting the above inequality \eqref{eq:key-integral-ineq-from-Sec3-in-GurHangLin}, Gursky, Hang and Lin fixed a constant $q\geq 0$ satisfying
\begin{align*}
\frac{n-2}{n-4} > q + 1 > \frac{4(n-1)}{n(n-4)}
\end{align*}
which is valid only for dimension $n \geq 6$.

\vspace*{1.5mm}
For our purpose in dimension $n = 5$, we make a different choice of $q \geq 0$. First, recall that $0< \lambda_{0} \leq \frac{14}{5}$ by the assumption \eqref{eq:extra-cond-on-initial-metric-for-nonempty-dim-5}. Next, observe that the following inequality holds for any $n \geq 5$:
\begin{align*}
\frac{n-2}{n-4} > \frac{14(n-1)}{5n(n-4)}.
\end{align*}
Now, we choose a constant $q\geq 0$ such that 
\begin{align}\label{eq:choice-of-q-in-dim-5}
\frac{n-2}{n-4} > q + 1 > \frac{14(n-1)}{5n(n-4)} \geq \frac{\lambda_{0}(n-1)}{n(n-4)}.
\end{align}
Note that, with such a choice for $q \geq 0$ satisfying \eqref{eq:choice-of-q-in-dim-5}, both the constants $A_{\alpha, q, n}$ and $B_{\alpha, q, n}$ are strictly positive for $\alpha$ large enough since
\begin{align*}
\lim_{\alpha \to \infty} A_{\alpha, q, n} &= (q+1) - \frac{(n-1)\lambda}{n(n-4)}; \\ \lim_{\alpha\to\infty} B_{\alpha, q, n} &= \frac{(q+1)(n-4)}{2}\left(\frac{n-2}{n-4} - (q+1)\right).
\end{align*}
Therefore, for large enough $\alpha \geq 1$ and with any constant $q\geq 0$ satisfying the condition \eqref{eq:choice-of-q-in-dim-5} we have the following key integral inequality
\begin{equation}\label{eq:key-int-ineq-after-making-a-choice-for-q}
\alpha \int_M v^{\alpha - 1} |\nabla v|^2\ d\mu + \int_M u^q v^{\alpha + 2}\ d\mu \leq c \int_M v^{\alpha + 1}\ d\mu + c\alpha\int_M v^{\alpha}\ d\mu
\end{equation}
for the function $v$ defined in \eqref{eq:aux-func-v-GurHangLin}. This is exactly the equation (3.39) in \cite{GurHangLin2016}.

\vspace*{1.5mm}
Since the crucial integral inequality \eqref{eq:key-int-ineq-after-making-a-choice-for-q} is established, the rest of the proof for higher-order estimates and for closedness of the set $\Sigma$ goes verbatim as in \cite{GurHangLin2016}. For instance, using \eqref{eq:key-int-ineq-after-making-a-choice-for-q} and the uniform positive lower bound $u \geq c > 0$ obtained in Lemma \ref{lem:unif-positive-lower-bdd-GurHangLin}, we first get the $L^{\infty}$ bound for the auxiliary function $v$. Combining equation \eqref{eq:2nd-order-equ-satisfied-by-aux-func-v} with the $L^{\frac{2n}{n-4}}$ bound from \eqref{eq:uniform-positive-lower-bdd-and-some-Lp-estm}, a standard bootstrap argument yields the $L^{\infty}$ bound for positive solutions $u$ to the equation \eqref{eq:GurHangLin-continuity-path-in-terms-of-lambda}. Consequently, we obtain $W^{2, p}$ estimates for any $1 < p < \infty$ and $C^{k}$ bounds for every $k\in \mathbb{N}$. As a consequence of these higher-order estimates, we conclude that the set $\Sigma$, defined in \eqref{eq:defn-of-the-set-Sigma}, is closed (see Proposition $3.1$ in \cite{GurHangLin2016} for more details).

\begin{proof}[\textbf{Proof of Theorem \ref{thm:GurHangLin-conj-dim-5-cont-method-under-an-extra-cond}}]
The positivity condition \eqref{eq:extra-cond-on-initial-metric-for-nonempty-dim-5} for an initial conformal metric $g_0 \in [g]$ implies that the set $\Sigma$, defined in \eqref{eq:defn-of-the-set-Sigma}, is non-empty. Moreover, assuming $Y(M, [g])>0$ and $Y^{\ast}_{4}(M, [g]) > 0$ we showed that $\Sigma$ is both open and closed. Hence we must have $\Sigma = [0, \lambda_{0}]$, i.e. $0 \in \Sigma$. This means that we have the existence of a conformal metric $\tilde g \in [g]$ with the $Q$-curvature $\tilde Q > 0$ and the scalar curvature $\tilde R > 0$. So, the proof is completed.
\end{proof}

\begin{proof}[\textbf{Proof of Corollary \ref{cor:positivity-of-Q-curv-assump-sigma-2-Yam-inv-positive}}]
Integrating \eqref{eq:Q-curv-defn-involving-sigma_2} yields
\begin{align*}
\int_M Q \ d\mu_{g} = \frac{n-4}{2} \int_M J^2 \ d\mu_{g} + 4 \int_M \sigma_{2}(A) \ d\mu_{g}.
\end{align*}
For $n \geq 5$, the definitions of $Y_{4}^{\ast}(M, [g])$ and $Y_{\sigma_{2}}(M, [g])$ imply that $$Y_{4}^{\ast}(M, [g]) \geq 2 (n-4) Y_{\sigma_{2}}(M, [g]).$$
Since $n=5$ and $Y_{\sigma_{2}}(M, [g]) > 0$, we have that $Y_{4}^{\ast}(M, [g]) > 0$. Thus, Theorem \ref{thm:GurHangLin-conj-dim-5-cont-method-under-an-extra-cond} completes the proof.
\end{proof}

\vspace*{2.5mm}
\section*{Acknowledgements}

The author extends sincere thanks to Prof. Saikat Mazumdar and Imran Hussain for insightful discussions and valuable feedback during the preparation of this paper. This work was supported in part by an Institute Post Doctoral Fellowship from the Indian Institute of Technology Bombay.

\vspace*{2mm}
\section*{Declarations}

\textbf{Data availability statement:} Data sharing not applicable to this article since no datasets were generated or analyzed during the current study.

\vspace*{1.5mm}
\textbf{Conflict of interest statement:} The author declare no conflict of interest.

\vspace*{2.5mm}

\end{document}